\documentclass[preprint,12pt]{elsarticle}
\usepackage{amsthm,amsmath,amssymb}
\usepackage[colorlinks=true,citecolor=black,linkcolor=black,urlcolor=blue]{hyperref}
\usepackage{graphicx}
\usepackage{float}
\usepackage{mathtools}
\usepackage{epstopdf}
\usepackage{epsfig}
\usepackage{subfigure}

\theoremstyle{plain}
\newtheorem{theorem}{Theorem}
\newtheorem{lemma}[theorem]{Lemma}
\newtheorem{corollary}[theorem]{Corollary}
\newtheorem{proposition}[theorem]{Proposition}

\theoremstyle{definition}
\newtheorem{definition}[theorem]{Definition}
\newtheorem{construction}[theorem]{Construction}
\newtheorem{example}[theorem]{Example}

\newtheorem{question}[theorem]{Question}
\theoremstyle{remark}
\newtheorem{remark}[theorem]{Remark}

\title{On the constant partial-dual polynomials of hypermaps}
\author{Yibing Xiang,
Qi Yan\footnote{Corresponding author.}\\
\small School of Mathematics and Statistics\\[-0.8ex]
\small Lanzhou University\\[-0.8ex]
\small P. R. China\\
\small{\tt Email:xiangyb2023@lzu.edu.cn; yanq@lzu.edu.cn}
}
\date{}
\journal{arXiv}

\begin{document}
\begin{abstract}
In this paper, we introduce the partial-dual polynomial for hypermaps, extending the concept from ribbon graphs. We discuss the basic properties of this polynomial and characterize it for hypermaps with exactly one hypervertex containing a non-zero constant term. Additionally, we show that the partial-dual polynomial of a prime connected hypermap $H$ is constant if and only if $H$ is a plane hypermap with a single hyperedge.
\end{abstract}
\maketitle

\section{Introduction}
Several methodologies exist for describing graphs embedded in surfaces, each of which has distinct advantages. For example, cellularly embedded graphs and their representations appear in the forms of maps, ribbon graphs, band decompositions, arrow presentations, and signed rotation systems \cite{EM}.

For any ribbon graph $G$, there is a natural dual ribbon graph $G^{*}$, also known as the classical Euler-Poincar\'{e} duality. Chmutov \cite{CG} introduced an extension to geometric duality called {\it partial duality}. Generally speaking, a partial dual $G^{A}$ is obtained by forming the classical Euler-Poincar\'{e} dual with respect to a subset $A \subseteq E(G)$ of the ribbon graph $G$.

In \cite{GMT}, Gross, Mansour and Tucker introduced the enumeration of the partial duals $G^A$ of a ribbon graph $G$, based on the Euler genus $\varepsilon$, over all edge subsets $A\subseteq E(G)$. The associated generating functions, denoted as $^{\partial}\varepsilon_{G}(z)$, are called \emph{partial-dual polynomials} of $G$. This polynomial has been investigated in numerous papers \cite{QCYC, SCFV, GMT, GMT2, GMT3, QYXJ, QYXJ2}.

Maps or ribbon graphs can be thought of as graphs embedded in surfaces. Hypermaps are hypergraphs embedded into surfaces. In other words, in hypermaps, a hyperedge is allowed to connect more than two hypervertices, thus having more than two half-edges, or even just a single half-edge. Chumutov and Vignes-Tourneret \cite{SCFV2} introduced the concept of partial duality for hypermaps, which also includes the classical Euler-Poincar\'{e} duality as a special case. Independently, Smith \cite{BSM} discovered this generalization.  Additionally, Chumutov and Vignes-Tourneret expressed partial duality in terms of permutational models and provided a formula for the Euler-genus change. They also posed the following question:

\begin{question}[\cite{SCFV}]
The paper \cite{GMT} contains several interesting conjectures about the partial-dual polynomial for ribbon graphs. The definition of this polynomial works for hypermaps, as well. It would be interesting to find (and prove) the analogous conjectures for hypermaps.
\end{question}

Combinatorially, hypermaps can be described in three distinct ways \cite{SCFV2}: firstly, as three involutions on the set of flags, which constitutes the bi-rotation system or the $\tau$-model; secondly, as three permutations on the set of half-edges, forming the rotation system or the $\sigma$-model in the orientable case; and thirdly, as 3-edge-coloured graphs. In this paper, we begin by representing hypermaps as ribbon hypermaps, similar to representing maps as ribbon graphs. However, our approach is somewhat intuitive, providing geometric reasoning rather than formal proofs for the equivalences. We show that the partial-dual polynomial has ribbon hypermap analogues. We introduce the partial-dual polynomial of ribbon hypermaps and discuss its basic properties. We characterize the partial-dual polynomial of a ribbon hypermap with exactly one hypervertex containing non-zero constant term. Finally, we show that the partial-dual polynomial of a prime connected ribbon hypermap $H$ is constant if and only if $H$ is a plane ribbon hypermap with a single hyperedge.

\section{Preliminaries}

\begin{definition}[\cite{BSM}]
 A \emph{ribbon hypermap} $H=(V(H), E(H))$ is a (possibly non-orientable) surface with boundary represented as the union of two sets of discs, a set $V(H)$ of hypervertices, and a set $E(H)$ of hyperedges such that
\begin{description}
  \item[(1)] The hypervertices and hyperedges intersect in disjoint line segments called {\it common line segments};
\item[(2)] Each such common line segment lies on the boundary of precisely one hypervertex and precisely one hyperedge.
\end{description}
\end{definition}

A {\it ribbon graph} is a ribbon hypermap in which every hyperedge contains precisely two common line segments.  A ribbon hypermap is {\it orientable} if it is orientable when viewed as a surface (here, we are using the fact that a ribbon hypermap is a punctured surface).  We adopt terminology and notation for ribbon hypermaps similar to those commonly used in the literature on ribbon graphs, including the following conventions.

\begin{description}
  \item[(C1)] If $H$ is a ribbon hypermap, then $V(H), E(H)$ and $F(H)$ denote, respectively, the sets of hypervertices, hyperedges, and hyperfaces of $H$. Correspondingly, 
$v(H), e(H)$, and $f(H)$  represent the cardinalities of these sets, i.e., the numbers of hypervertices, hyperedges, and hyperfaces in $H$, respectively.
The notation $k(H)$ is used to denote the number of connected components of $H$. Moreover,  for $A\subseteq E(H)$, we let $A^c = E(H)-A$.

\item[(C2)] We view $A\subseteq E(H)$ also as the {\it sub-ribbon hypermap} obtained from a ribbon hypermap $H$ by removing from $H$ all hyperedges not in the set $A$; note that although we delete the hyperedges, we retain the hypervertices, even if they become isolated.

  \item[(C3)] Given a ribbon hypermap $H$ and a hyperedge $e\in E(H)$, the {\it degree} $d(e)$ of $e$ is defined as the number of half-edges (common line segments) incident to $e$. For any subset $A\subseteq E(H)$, we denote the total degree of hyperedges in $A$ by $d(A)=\sum_{e\in A}d(e)$. Furthermore, we define the total degree of $H$ as $d(H):=d(E(H))$.
  
  \item[(C4)] For any ribbon hypermap 
$H$, we define the {\it Euler characteristic}
$\chi(H)$, regardless of whether  $H$ is connected or not, as follows:
\[ \chi(H)=v(H)+e(H)+f(H)-d(H). \]
  
  \item[(C5)] The {\it Euler genus} 
$\varepsilon(H)$ of a ribbon hypermap 
$H$ is defined using Euler characteristic as follows:
$$\varepsilon(H)=2k(H)-\chi(H)=2k(H)+d(H)-v(H)-e(H)-f(H).$$
In particular, if $H$ is a ribbon graph, the Euler genus simplifies to: $$\varepsilon(H)=2k(H)-v(H)+e(H)-f(H).$$

\item[(C6)] A ribbon hypermap $H$
is a {\it plane ribbon hypermap} if it is connected and $\varepsilon(G) = 0$.

\item[(C7)] A {\it hyper-bouquet} is a ribbon hypermap with exactly one hypervertex. A {\it hypertree} is a ribbon hypermap $H$ with $f(H)=1$ and $\varepsilon(H)=0$. A {\it hyper-quasi-tree}, as a generalization of a hypertree, is a ribbon hypermap $H$ with $f(H)=1$, relaxing the condition $\varepsilon(H)=0$.

\end{description}

\begin{definition}[\cite{BSM}]
An \textit{arrow presentation} of a ribbon hypermap consists of
a set of closed curves, which represent the hypervertices, and sets of labeled arrows which represent the hyperedges. The set of arrows is partitioned into cyclically ordered subsets such that each subset represents a single hyperedge. We label the arrows that represent a hyperedge 
$e$ by $e_1,\cdots, e_{d(e)}$, where the indices give their cyclic order.
\end{definition}

A ribbon hypermap can be obtained from an arrow presentation as follows. Fill in the closed curves to form the hypervertex discs. Then, for each cyclically ordered subset constituting the partition of the arrows, draw a line segment from the head of one arrow to the tail of the next arrow in the cyclic order. The union of these arrows and line segments now forms a closed curve. Attach a disc along this closed curve to represent a hyperedge. Conversely, an arrow presentation can be derived from a ribbon hypermap as follows: For each hyperedge
$e$ of the ribbon hypermap, draw an arrow along one of its common line segments and label this arrow $e_1$ (note that the direction of the first arrow is arbitrary). Then, follow the boundary of the hyperedge from the head of the arrow to the next hypervertex (which may be the same hypervertex if a hyperedge intersects a single hypervertex multiple times), and draw another arrow along the next common line segment, starting at the point where the boundary meets the vertex, and label this new arrow $e_2$. Repeat this process until you return to the original common line segment. Perform this procedure for all other hyperedges. Finally, delete all the hyperedges and replace the vertex discs with closed curves to obtain an arrow presentation.

Each ribbon hypermap has a naturally corresponding ribbon graph, as follows.

\begin{construction}\label{cons01}
Let $H$ be a ribbon hypermap. For each hyperedge $e\in E(H)$, replace $e$ with a new hypervertex,
called the {\it central vertex} of $e$ and denoted by $v_e$, and introduce $d(e)$ new hyperedges, labeled $e_1, \cdots, e_{d(e)}$, each having two half-edges and incident to $v_e$ in a specific configuration, as illustrated in Figure \ref{f02}. Applying this transformation to all hyperedges yields a ribbon graph, denoted by $R(H)$.
\end{construction}

\begin{figure}[htbp]
\centering
\includegraphics[width=13cm]{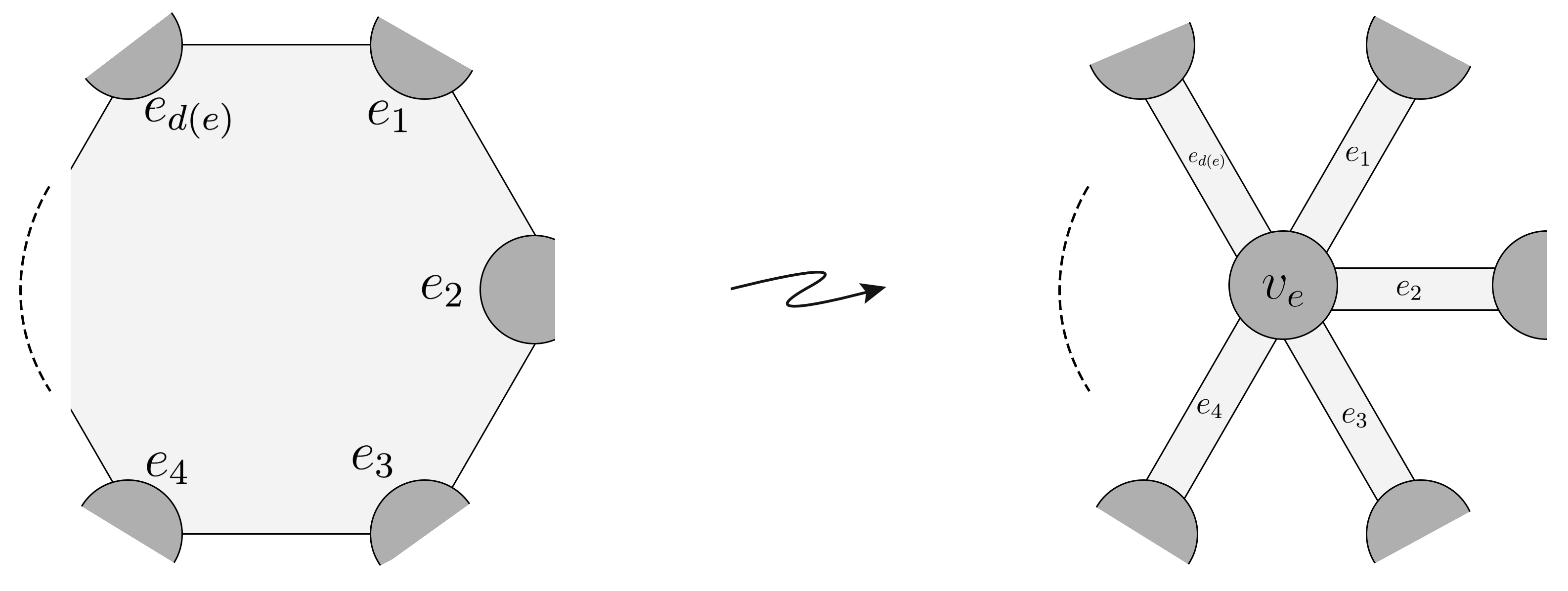}
\caption{Construction \ref{cons01}}
\label{f02}
\end{figure}
	
\begin{remark}
\begin{description}
  \item[(1)]
Given the Euler's formulas for a hypermap $H$ and its corresponding ribbon graph 
$R(H)$:
$$\varepsilon(H)=2k(H)+d(H)-v(H)-e(H)-f(H)$$ and $$\varepsilon(R(H))=2k(R(H))-v(R(H))+e(R(H))-f(R(H)),$$ 
we observe that $\varepsilon(H)=\varepsilon(R(H))$. This equality holds because of the following relationships: $k(R(H))=k(H)$, $v(R(H))=v(H)+e(H)$, $e(R(H))=d(H)$, and $f(R(H))=f(H)$. 

\item[(2)] It is evident that deleting a hyperedge $e$ from a ribbon hypermap $H$ cannot increase its Euler genus. This operation is equivalent to removing the vertex $v_e$ and the edges $e_1,\cdots, e_{d(e)}$ in $R(H)$.
\end{description}
\end{remark}

\begin{definition}
Let $H$ be a ribbon hypermap and let $v\in V(H)$. Suppose that $cl_1, cl_2, cl_3$, and $cl_4$ are  common line segments of hyperedges $e, f, g$, and $h$ respectively. Here, it is possible that $e=g$ and/or $f=h$. We say that the four common line segments $cl_1, cl_2, cl_3$, and $cl_4$ are {\it alternate} at $v$ if there exists a cyclic order $(cl_{1}\cdots  cl_{2}\cdots  cl_{3}\cdots  cl_{4}\cdots)$ around the boundary of $v$ such that the vertices $v_{e}$ and $v_{g}$  (the central vertices of $e$ and $g$) lie in one connected component of $R(H)\setminus v$, and the vertices $v_{f}$ and $v_{h}$ (the central vertices of $e$ and $g$) lie in a different connected component of $R(H)\setminus v$.
\end{definition}

\begin{example}
In Figure \ref{ex2}, we present two ribbon hypermaps where the four common line segments $cl_{1}$, $cl_{2}$, $cl_{3}$ and $cl_{4}$ alternate.
\begin{figure}[htbp]
    \centering
    \includegraphics[width=13cm]{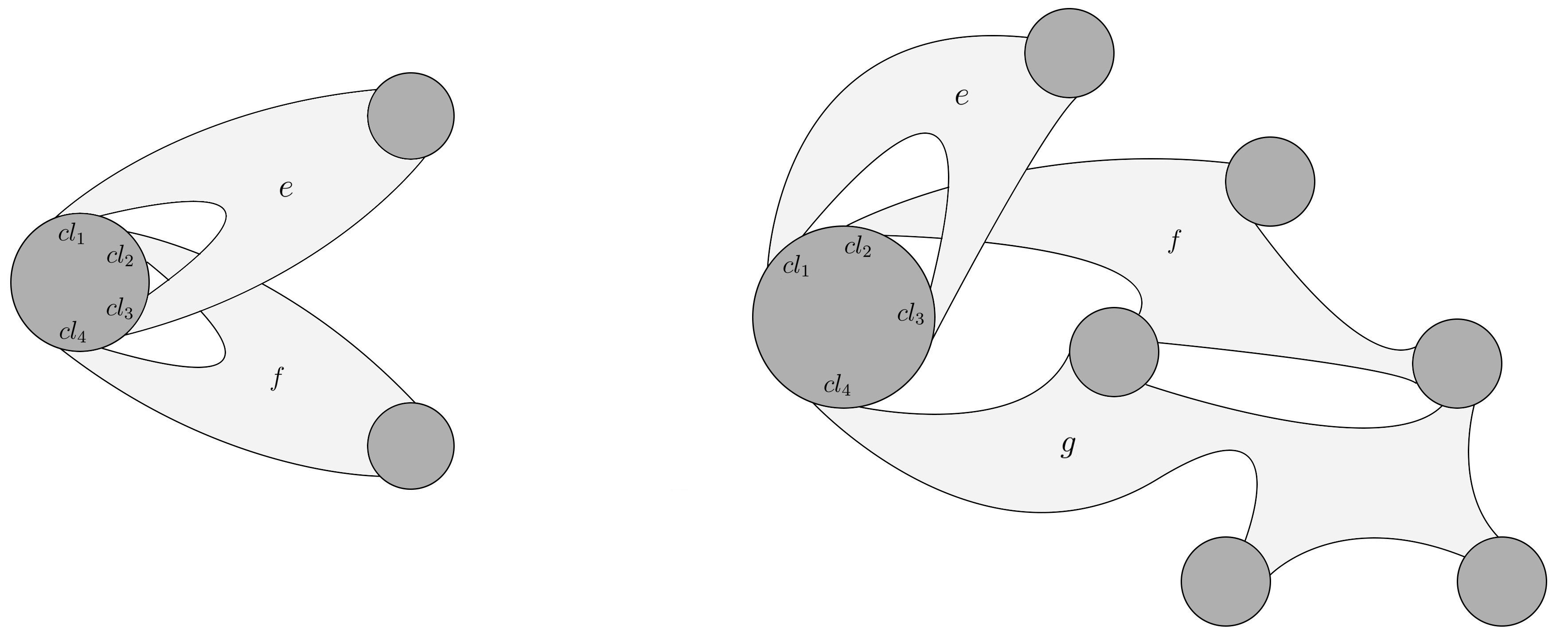}
    \caption{The four common line segments $cl_1, cl_2, cl_3$, and $cl_4$ alternate}
    \label{ex2}
\end{figure}
\end{example}

\section{Partial duality for ribbon hypermaps}

\begin{definition}[\cite{BSM}]
Let $H$ be an arrow presentation and $A\subseteq E(H)$. The {\it partial dual}, 
$H^A$, of $H$ with respect to 
$A$ is constructed as follows.
For each edge $e\in A$, in the arrow presentation of $H$, there are $d(e)$ labeled arrows cyclically ordered as $e_1, \cdots , e_{d(e)}$. For each $i$ with $1\leq i\leq d(e)-1$,
we draw a line segment with an arrow from the head of $e_i$ to the tail of $e_{i+1}$ and label this newly created arrow as $e_i$. When $i=d(e)$, we draw an arrow from the head of $e_{d(e)}$ to the tail of $e_1$ and label it $e_{d(e)}$. 
Then we delete the original arrows. The new arrows become arcs of new closed curves in the arrow presentation of $H^A$. 
\end{definition}

\begin{remark}\label{re01}
According to \cite{SCFV2}, partial duality preserves orientability of ribbon hypermaps.
\end{remark}

\begin{theorem}\label{the01}
 Let $H$ be a ribbon hypermap and $A\subseteq E(H)$. Then
$$\chi(H^{A})=\chi(A)+\chi(A^{c})-2v(H).$$
\end{theorem}

\begin{proof}
Given the Euler characteristic of a ribbon hypermap 
$H^{A}$ as: \[\chi(H^{A})=v(H^{A})+e(H^{A})+f(H^{A})-d(H^{A}),\]
we substitute $v(H^{A})=f(A)$, $e(H^{A})=e(H)$, and $d(H^{A})=d(H)$, yielding
		\[\chi(H^{A})=f(A)+e(H)+f(H^{A})-d(H).  \]
By the definition of the partial dual, we have $$f(H^{A})=v((H^{A})^{*})=v(H^{A^{c}})=f(A^{c}).$$ 
Using the fact that $f(A)=\chi(A)-v(A)-e(A)+d(A)$ and $f(A^{c})=\chi(A^{c})-v(A^{c})-e(A^{c})+d(A^{c})$, and noting that  $v(A)=v(A^{c})=v(H)$, we can rewrite the expression for 
$\chi(H^{A})$ as:
\begin{eqnarray*}
\chi(H^{A})&=&(\chi(A)-v(A)-e(A)+d(A))+e(H)+\\
&~&(\chi(A^{c})-v(A^{c})-e(A^{c})+d(A^{c}))-d(H)\\
		&=&\chi(A)-v(A)+\chi(A^{c})-v(A^{c})\\
&=&\chi(A)+\chi(A^{c})-2v(H).
		\end{eqnarray*}
	\end{proof}

\begin{corollary}\label{cor04}
Let $H$ be a ribbon hypermap and $A\subseteq E(H)$. Then
\[ \varepsilon(H^{A})=\varepsilon(A)+\varepsilon(A^{c})+2(k(H)-k(A)-k(A^{c}))+2v(H). \]
\end{corollary}

\begin{proof}
We first make some substitutions based on the formula for $\varepsilon(H)$:
\[ \chi(H^{A})=2k(H^{A})-\varepsilon(H^{A});\ \chi(A)=2k(A)-\varepsilon(A);\ \chi(A^{c})=2k(A^{c})-\varepsilon(A^{c}). \]
According to Theorem \ref{the01}, we have
\[ \varepsilon(H^{A})=\varepsilon(A)+\varepsilon(A^{c})+2(k(H)-k(A)-k(A^{c}))+2v(H). \]
\end{proof}

\begin{corollary}\label{cor01}
Let $H$ be a hyper-bouquet and $A\subseteq E(H)$. Then
\[ \varepsilon(H^{A})=\varepsilon(A)+\varepsilon(A^{c}). \]
\end{corollary}

\begin{proof}
By Corollary \ref{cor04}, replacing $$k(A)=k(A^{c})=k(H)=k(H^{A})=v(H)=1$$ immediately yields $$\varepsilon(H^{A})=\varepsilon(A)+\varepsilon(A^{c}).$$
\end{proof}

\begin{proposition}
Let $H$ be a ribbon hypermap and $e\in E(H)$. Then
\[ \lvert \varepsilon(H)-\varepsilon(H^{e})\rvert\leqslant2(d(e)-1). \]\end{proposition}
	
\begin{proof}
Starting with the Euler's formula
$$\varepsilon(H)=2k(H)+d(H)-v(H)-e(H)-f(H)$$ and $$\varepsilon(H^e)=2k(H^e)+d(H^e)-v(H^e)-e(H^e)-f(H^e).$$
Since the operation of partial dual does not change the number of half-edges and connected components, and given that $v(H^{e})=f(e)$, $e(H^{e})=e(H)$ and $f(H^{e})=f(H-e)$, we have $$\varepsilon(H)-\varepsilon(H^e)=v(H^{e})-v(H)+f(H^{e})-f(H)=f(e)-v(H)+f(H-e)-f(H).$$
By Euler's formula, we have $$\varepsilon(e)=2k(e)+d_{e}-v(e)-1-f(e).$$ Then we get $$f(e)=2k(e)+d_{e}-v(e)-1-\varepsilon(e)\leq 2k(e)+d_{e}-v(H)-1.$$
We observe that deleting a hyperedge $e$ will increase the number of boundary components by at most $d_{e}-1$ if the half-edges of $e$ are incident to distinct boundary components of $H-e$.
Hence, 
\begin{eqnarray*}
\varepsilon(H)-\varepsilon(H^e)
&=&v(H^e)-v(H)+f(H^e)-f(H)\\
&=&f(e)-v(H)+f(H-e)-f(H)\\
&\leq&2k(e)+d_{e}-v(H)-1-v(H)+d_{e}-1\\
&=&2(k(e)-v(H))+2(d_{e}-1)\\
&\leq&2(d_{e}-1).
\end{eqnarray*}
Moreover, we claim that $$\varepsilon(H)-\varepsilon(H^e)\geq 2(1-d_{e}).$$ If $\varepsilon(H)-\varepsilon(H^e)<2(1-d_{e})$, then $$\varepsilon(H^e)-\varepsilon((H^e)^e)=\varepsilon(H^e)-\varepsilon(H)>2(d_{e}-1),$$ which is a contradiction.
\end{proof}

\begin{remark}\label{remark1}
The upper bound is tight. For the ribbon hypermap $H$ as illustrated in Figure \ref{11}, we have $$\varepsilon(H)=2k(H)+d(H)-v(H)-e(H)-f(H)=2+2i-1-2-1=2i-2$$ and by Corollary \ref{cor01}, $\varepsilon (H^{e})=\varepsilon (e)+\varepsilon (f)=0$.
Therefore, \[\lvert \varepsilon(H)-\varepsilon(H^e)\rvert=2(i-1)=2(d(e)-1).\]
\begin{figure}[htbp]
		\centering
	\includegraphics[width=4cm]{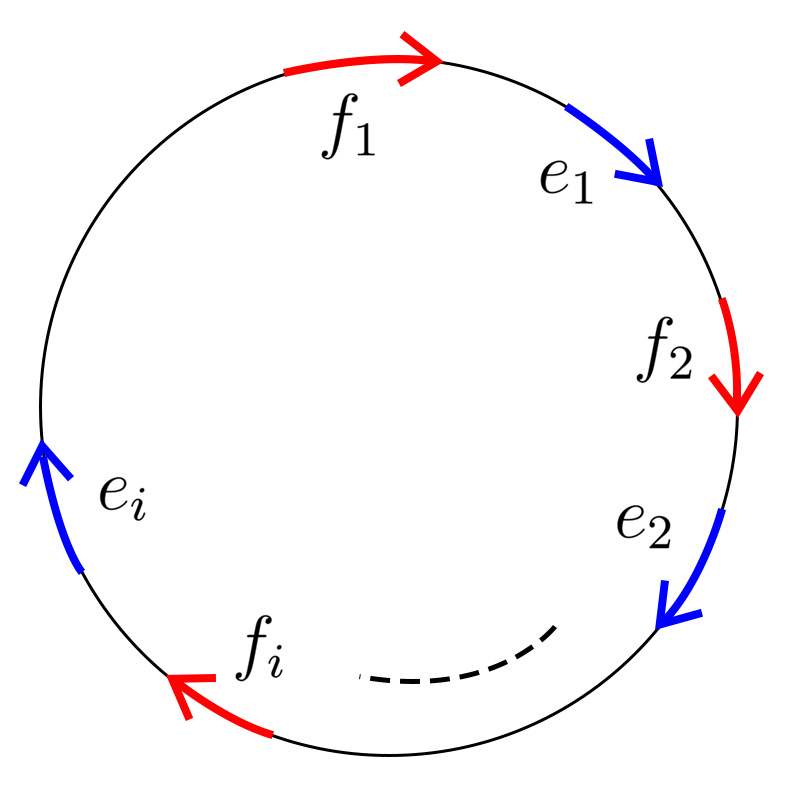}
		\caption{Remark \ref{remark1}}
        \label{11}
	\end{figure}
\end{remark}

\section{The partial-dual polynomial}
\begin{definition}
The \emph{partial-dual polynomial} of any ribbon hypermap $H$ is the generating function
$$^{\partial}\varepsilon_{H}(z):=\sum_{A\subseteq E(H)}z^{\varepsilon(H^A)}$$
that enumerates all partial-duals of $H$ by Euler-genus.
\end{definition}

\begin{proposition}\label{pro00}
Let H be a ribbon hypermap and $A\subseteq E(H)$. Then
\begin{description}
\item[(1)] $^{\partial}\varepsilon_{H}(1)=2^{e(H)}$.
\item[(2)] $^\partial{\varepsilon_{H}}(z)={^\partial{\varepsilon_{H^{A}}}}(z).$
\item[(3)] The degree of $^{\partial}\varepsilon_{H}(z)$ is at most $d(H)-e(H)$.
\end{description}
\end{proposition}

\begin{proof}
For (1), the evaluation $^{\partial}\varepsilon_{H}(1)$ counts the total number of partial duals, which is given by $2^{e(H)}$.
For (2),  since the sets of all partial duals of $H$ and $H^{A}$ are identical, the equation holds.
For (3),  
for any $A\subseteq E(H)$,
by Euler's formula, we have $$\varepsilon(H^{A})=2k(H^{A})+d(H)-v(H^{A})-e(H)-f(H^{A}).$$ Since $k(H^{A})\leqslant v(H^{A})$ and $k(H^{A})\leqslant f(H^{A})$, it follows that $$\varepsilon(H^{A})\leqslant d(H)-e(H).$$
\end{proof}

\begin{remark}\label{remark2}
For (3), the upper bound is tight. For the ribbon hypermap $H$ as illustrated in Figure \ref{ex1}, we have $d(H)=9$, $e(H)=3$ and $^\partial {\varepsilon_{H}}(z)=6z^{4}+2z^{6}$. The degree of $^{\partial}\varepsilon_{H}(z)$ equals $d(H)-e(H)=9-3=6$.
\begin{figure}[htbp]
    \centering
    \includegraphics[width=5cm]{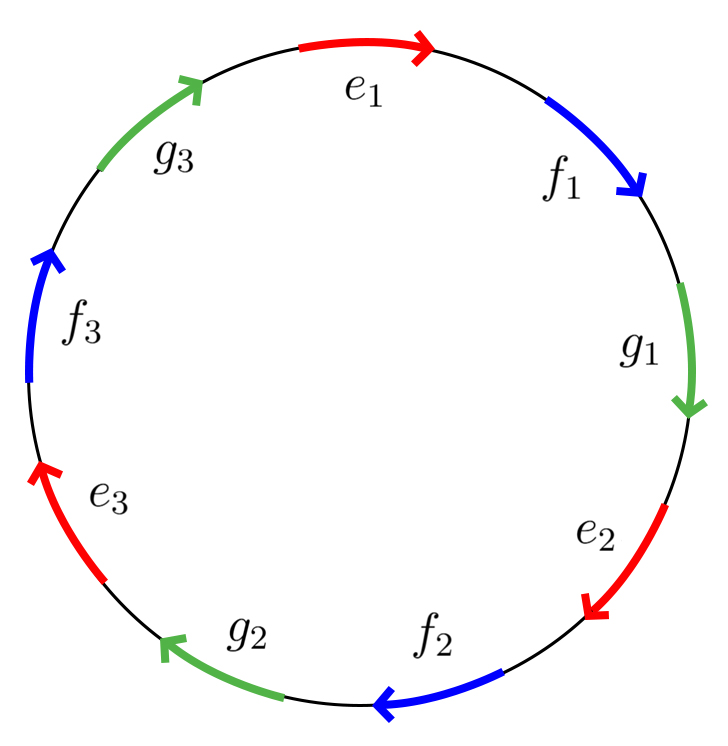}
    \caption{Remark \ref{remark2}}
    \label{ex1}
\end{figure}
\end{remark}

There are ways to combine two ribbon hypermaps $H_1$ and $H_2$ into a graph $H$ whose partial-dual polynomial is the product $^\partial{\varepsilon_{H_{1}}}(z){^\partial{\varepsilon_{H_{2}}}}(z)$. The {\it disjoint union} of $H_1$ and $H_2$, denoted by $H_{1}\cup H_{2}$, is formed by taking the unions of their hypervertices,  hyperedges, and hyperfaces.
We define the {\it one-vertex-joint} operation on two disjoint ribbon hypermaps $H_{1}$ and $H_{2}$, denoted by $H_{1}\vee H_{2}$, in two steps:

\begin{description}
  \item[(1)] Choose an arc $p_{1}$ on the boundary of a hypervertex $v_{1}$ of $H_{1}$ that lies between two consecutive hyperedge ends, and choose another such arc $p_{2}$ on the boundary of a hypervertex $v_{2}$ of $H_{2}$.
  \item[(2)] Paste the hypervertices $v_{1}$ and $v_{2}$ together by identifying the arcs $p_{1}$ and $p_{2}$.
\end{description}

\begin{proposition}\label{pro01}
Let $H_{1}$ and $H_{2}$ be disjoint ribbon hypermaps. Then
\[ ^\partial{\varepsilon_{H_{1}\cup H_{2}}}(z)={^\partial{\varepsilon_{H_{1}\vee H_{2}}}}(z)={^\partial{\varepsilon_{H_{1}}}}(z){^\partial{\varepsilon_{H_{2}}}}(z). \]
\end{proposition}

\begin{proof}
The underlying phenomenon is the additivity of Euler-genus over the two operations. We
observe that for any $A\subseteq E(H_{1})\cup E(H_{2})$, we have
$$(H_{1}\cup H_{2})^{A}=H_{1}^{A\cap E(H_{1})}\cup H_{2}^{A\cap E(H_{2})}$$
and
$$(H_{1}\vee H_{2})^{A}=H_{1}^{A\cap E(H_{1})}\vee H_{2}^{A\cap E(H_{2})}, $$
From these equations,  the result follows.
\end{proof}

\begin{lemma}\label{lem00}
For a ribbon hypermap $H$, if there exist four common line segments that alternate, then $\varepsilon (H)>0$.
\end{lemma}
	
\begin{proof}
We consider the corresponding ribbon graph $R(H)$ of $H$. If there exist four common segments  $cl_1, cl_2, cl_3$, and $cl_4$ which alternate, we assume that these segments are common line segments of hyperedges $e, f, g$, and $h$, respectively. Then there exists a cyclic order $(cl_{1}\cdots  cl_{2}\cdots  cl_{3}\cdots  cl_{4}\cdots)$ around the boundary of a hypervertex $v$ of $H$ such that the vertices $v_{e}$ and $v_{g}$ lie in one connected component of $R(H)\setminus v$, and the vertices $v_{f}$ and $v_{h}$ lie in a different connected component of $R(H)\setminus v$. Consequently, there exist two disjoint paths with endpoints $v_{e}, v_{g}$ and $v_{f}, v_{h}$, respectively. Hence, there is a sub-ribbon graph of $R(H)$, denoted by $SR(H)$, as shown in Figure \ref{sR(H)}. It follows that $$\varepsilon(H)=\varepsilon(R(H))\geqslant\varepsilon(SR(H))>0.$$
\begin{figure}[htbp]
    \centering
    \includegraphics[width=10cm]{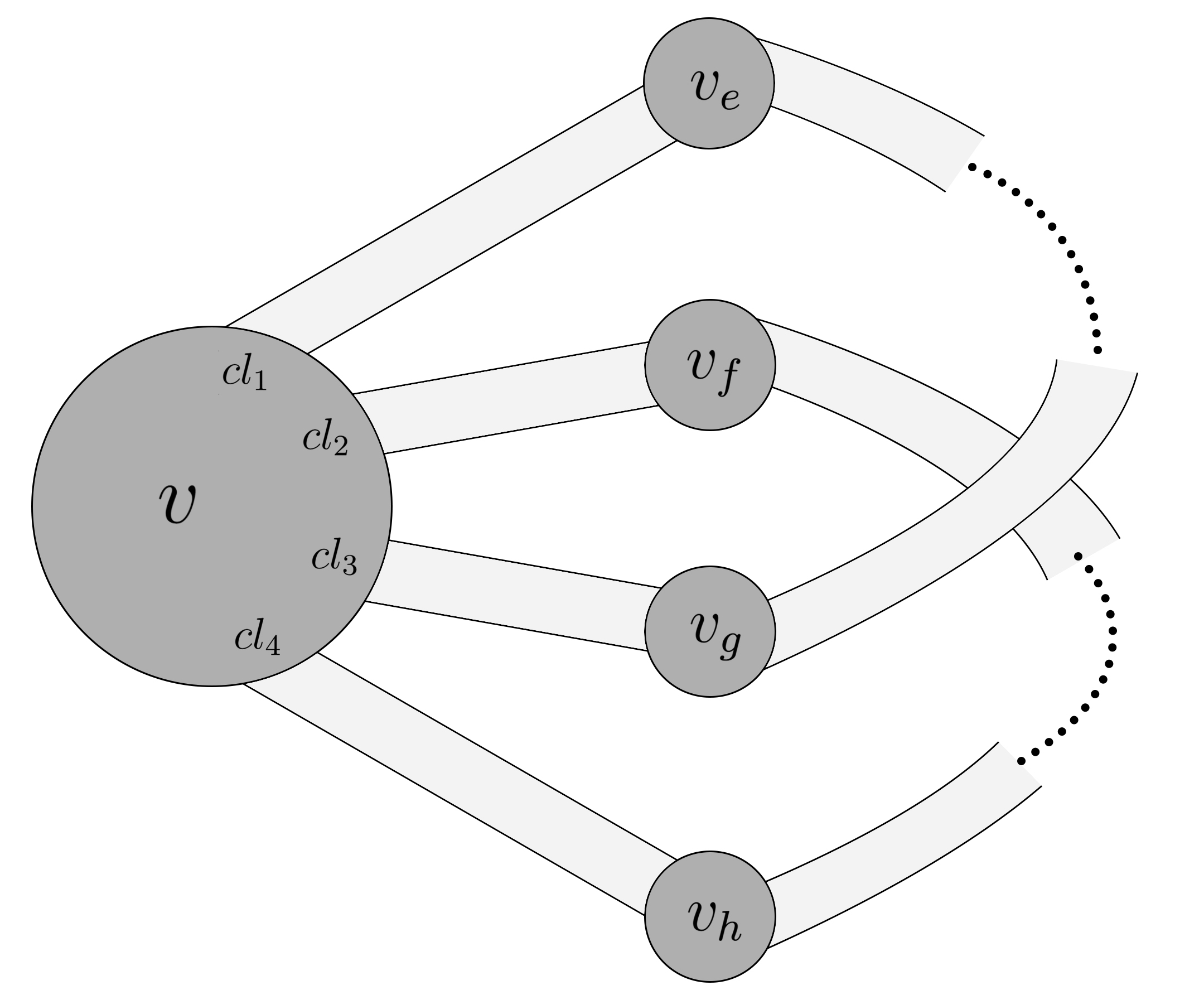}
    \caption{SR(H)}
    \label{sR(H)}
\end{figure}
\end{proof}

\begin{definition}
The intersection graph $I(B)$ of a hyper-bouquet $B$
is defined as follows: The vertex set of $I(B)$ is 
$E(B)$, and two vertices $e$ and 
$f$ in $I(B)$ are adjacent if and only if there exist four common line segments, denoted by $e_i, f_m, e_j$, and $f_n$, that alternate at the unique vertex of $B$. Here, $e_i$ and $e_j$ are two common line segments of $e$, and 
$f_m$ and $f_n$ are two common line segments of $f$. See Figure \ref{19} for an example.
\end{definition}
\begin{figure}[htbp]
\centering

\includegraphics[width=10cm]{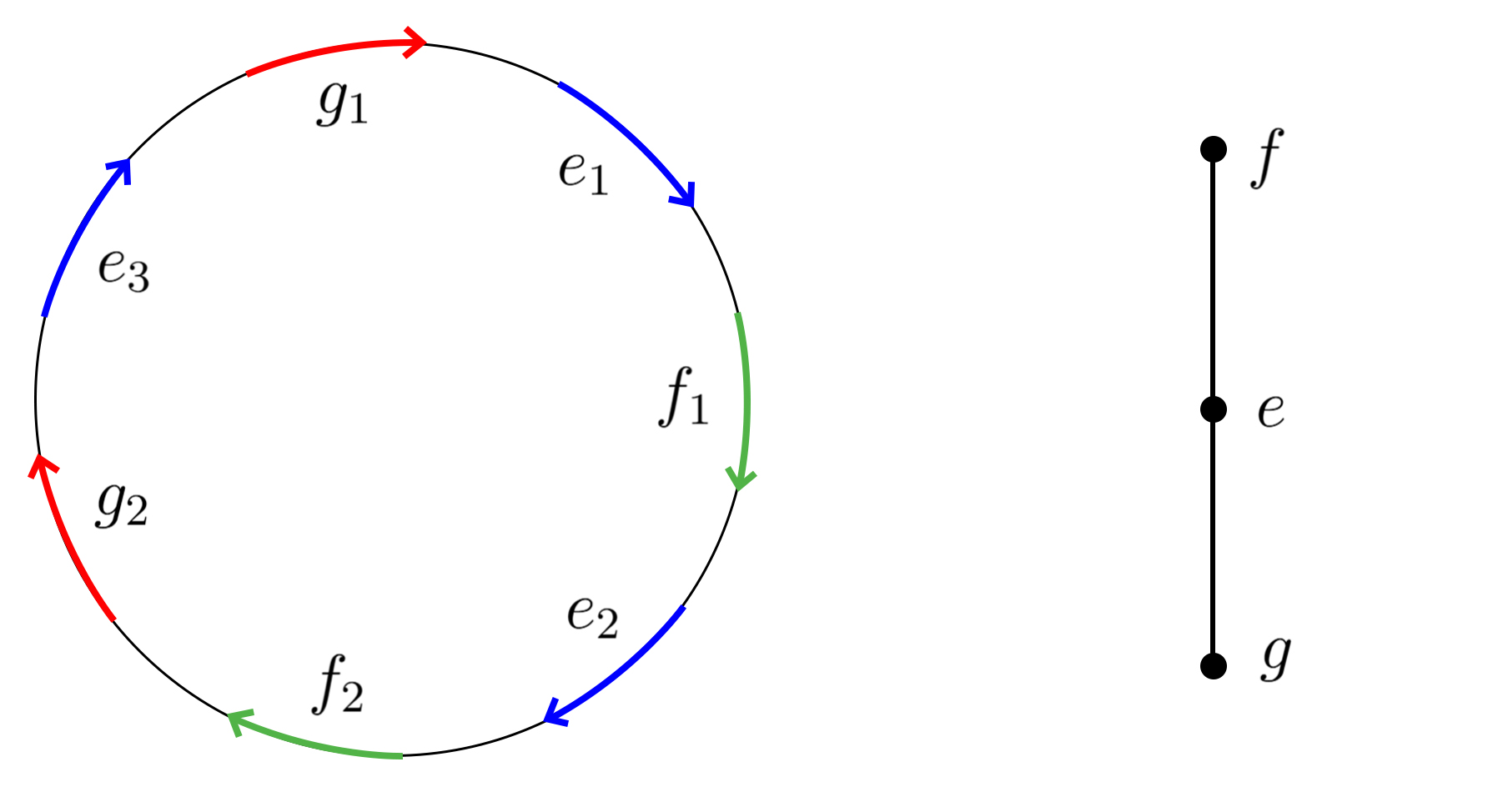}
\caption{A hyper-bouquet and its intersection graph}
\label{19}
\end{figure}
	
\begin{theorem}
The partial-dual polynomial $^\partial{\varepsilon_{(B)}}(z)$ of a hyper-bouquet $B$ contains a nonzero constant term if and only if $I(B)$ is bipartite and $\varepsilon(e)=0$ for all $e\in E(B)$.
\end{theorem}
	
\begin{proof}
Since $^\partial{\varepsilon_B}(z)$ contains a nonzero constant term, it follows  that $B$ is a partial dual of a plane ribbon hypermap.  Assume that 
$I(B)$ is not bipartite. Then $I(B)$ contains an odd cycle $C$. Let $D$ be the subset of hyperedges of 
$B$ corresponding to the vertices of $C$. Clearly, removing hyperedges cannot increase the Euler genus.
Hence, for any subset $A$ of $E(B)$, we have $\varepsilon (A\cap D)\leqslant\varepsilon (A)$ and $\varepsilon (A^{c}\cap D)\leqslant\varepsilon (A^{c})$. Since $I(B)$ contains the odd cycle $C$, 
there exist two adjacent vertices when the vertices of $C$ are partitioned into two sets. Denote the hyperedges corresponding to these two vertices by $e$ and $f$. Then either both $e$ and  $f$ are in $A\cap D$ or both are in $A^{c}\cap D$ such that their half-edges meet in the cyclic order 
$(e_{i}\cdots f_{m}\cdots e_{j}\cdots f_{n}\cdots)$
when traversing the boundary of the unique vertex of $B$. Hence,  the four common line segments $e_{i}, f_{m}, e_{j}$, and $f_{n}$ alternate. Therefore, by Lemma \ref{lem00}, we have $\varepsilon (A\cap D)+\varepsilon (A^{c}\cap D)>0$. Consequently, $\varepsilon(B^{A})=\varepsilon(A)+\varepsilon(A^{c})>0$ by Corollary \ref{cor01}.
However, since $B$ is a partial dual of a plane ribbon hypermap, there exists $A^{'}\subseteq E(B)$ such that $\varepsilon (B^{A^{'}})=0$, which leads to a contradiction. If there exists a hyperedge $h$ such that $\varepsilon(h)>0$, then for any $M\subseteq E(B)$ we have $\varepsilon (B^{M})=\varepsilon (M)+\varepsilon (M^{c})\geqslant\varepsilon (h)>0$, which contradicts that $^\partial{\varepsilon_{(B)}}(z)$ contains a nonzero constant term. Hence, $\varepsilon(e)=0$ for all $e\in E(B)$.

Conversely, if $I(B)$ is bipartite and $\varepsilon(e)=0$ for all $e\in E(B)$, then its vertex set can be partitioned into two subsets $X$ and $X^c$ such that every edge of $I(B)$ has one end in $X$ and the other end in $X^c$. For these two subsets $X$ and $X^c$ of the vertex set of $I(B)$, we also denote these two corresponding hyperedge subsets of $B$ by $X$ and $X^c$. Obviously, the hyperedges of $X$ do not intersect each other. Let $X=\{e_1, \cdots, e_{|X|}\}$. Then $X=e_1\vee \cdots \vee e_{|x|}$. Hence, $\varepsilon (X)=\varepsilon(e_1)+\cdots+\varepsilon(e_{|x|})=0$.
Similarly, for $X^c$, we have $\varepsilon (X^c)=0$.
Thus, $\varepsilon (B^{X})=\varepsilon (X)+\varepsilon (X^c)=0$ by Corollary \ref{cor01}. Therefore, $^\partial{\varepsilon_B}(z)$ contains a nonzero constant term.
\end{proof}

We say that a ribbon hypermap $H$ is {\it prime}, if there do not exist non-empty sub-ribbon hypermaps $H_1, \cdots, H_k$ of $H$  such that $H=H_1 \vee \cdots \vee H_k$ where $k\geqslant 2$.
	
\begin{theorem}\label{th00}
Let $H$ be a prime connected ribbon hypermap. Then $^\partial{\varepsilon_{H}}(z)=k$ for some integer $k$ if and only if $H$ is a plane ribbon hypermap with $e(H)=1$.
\end{theorem}

\begin{proof}
For sufficiency, if $H$ is a plane ribbon hypermap with $e(H)=1$, then we 
 simply obtain that $^\partial{\varepsilon_{H}}(z)=2$.

For necessity, we know that $\varepsilon(H)=0$ by $^\partial{\varepsilon_{H}}(z)=k$. Suppose that $H$ is prime, connected  and $e(H)\geqslant2$. Let 
$cl_{1}$ and $cl_{2}$ be two common line segments of hyperedges $e$ and $f$ (it is possible that $e=f$), respectively. And $cl_{1}$ and $cl_{2}$  lie on the boundaries of the hypervetrices $v_1$ and $v_2$ (it is possible that $v_1=v_2$), respectively. We define that the two common line segments $cl_{1}$ and $cl_{2}$ are connected if $v_e$ and $v_f$ (the central vertices of $e$ and $f$) are in the same connected component of $R(H)\setminus\{v_1, v_2\}$. 

Let $g$ be a hyperedge of $H$ and let $H|_{g}$ be the induced sub-ribbon hypermap of the hyperedge $g$ in $H$.
We claim that for any two connected common line segments belonging to $E(H)\setminus g$, then these two segments lie on exactly the same boundary  component of $H|_{g}$. If there exist two common line segments $s_{1}$ and $s_{2}$ of $E(H)\setminus g$ lie on different boundary components of $H|_{g}$, then we have a connected sub-ribbon graph of $R(H)$, denoted by $SR(H)$, as shown in Figure \ref {R(H')}. Note that 
$$\varepsilon(R(H|_{g}))\leqslant\varepsilon (R(H))=\varepsilon(H)=0.$$ Hence $\varepsilon(R(H|_{g}))=0$. 
By Euler's formula, 
$$\varepsilon(R(H|_{g}))=2-\chi(R(H|_{g}))=2-((v(H|_{g})+1)-d(g)+f(H|_{g}))=0.$$ So we have 
\begin{eqnarray*}
\varepsilon(SR(H))&=&2-\chi(SR(H))\\
&=&2-((v(H|_{g})+1+i)-(d(g)+(i+1))+(f(H|_{g})-1))\\
&=&2>0,    
\end{eqnarray*}
where $i$ is the number of the vertices of a path connecting the central vertices corresponding $s_{1}$ and $s_{2}$. It contradicts $\varepsilon(SR(H))=\varepsilon(R(H))=\varepsilon(H)=0$.
So our claim is established.
\begin{figure}[htbp]
    \centering
    \includegraphics[width=13cm]{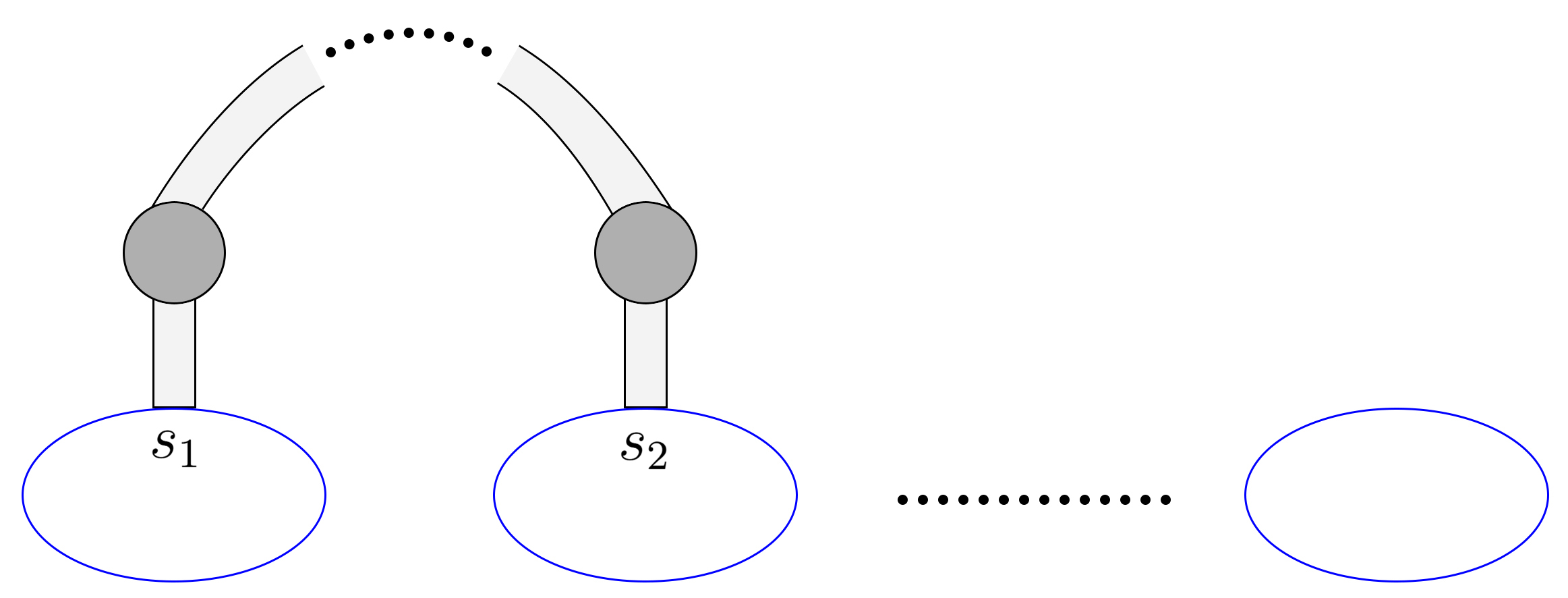}
    \caption{$SR(H)$}
    \label{R(H')}
\end{figure}
		
For $H^{g}$, each boundary component of $H|_{g}$ is a hypervertex of $H^{g}$. By the claim, there exist two connected common line segments $cl_1$ and $cl_2$ in $E(H^{g})\setminus g$, which lie on the boundary of a hypervertex $v$ of $H^{g}$. Moreover, there are two common line segments $g_{i}$ and $g_{j}$ of $g$ such that $cl_1, g_{i}, cl_2$ and $g_{j}$ alternate at $v$. Thus, by Lemma \ref{lem00}, we have $\varepsilon(H^{g})>0$ , which contradicts $^\partial{\varepsilon_{H}}(z)=k$.  Therefore, $H$ is a plane ribbon hypermap with $e(H)=1$.
\end{proof}

\begin{proposition}\label{pro06}
Let $H$ be a ribbon hypermap. Then $H$ is a plane hyper-bouquet if and only if $H^{*}$ is a hypertree. 
\end{proposition}

\begin{proof}
For sufficiency, since $H^{*}$ is a hypertree, we have  $f(H^{*})=1$ and $\varepsilon(H^{*})=0$. Consequently,
$v(H)=v((H^{*})^{*})=f(H^{*})=1$ and $\varepsilon(H)=\varepsilon(H^{*})=0$. Therefore, $H$ is a plane hyper-bouquet. Conversely, if $H$ is a plane hyper-bouquet, then $f(H^*)=v(H)=1$ and
$\varepsilon(H^*)=\varepsilon(H)=0$. Thus, $H^{*}$ is a hypertree.
\end{proof}

\begin{proposition}\label{pro05}
Let $B$ be a plane hyper-bouquet. Then $^\partial{\varepsilon_{B}}(z)=2^{e(B)}$.
\end{proposition}

\begin{proof}
Since $B$ is a plane hyper-bouquet, it follows that $B$ is the one-vertex-join of $e(B)$ hyperedges. Otherwise, there would exist four common line segments alternating at the unique vertex of $B$ and, by Lemma \ref{lem00}$, 
\varepsilon (B)>0$, which contradicts $\varepsilon (B)=0$. 
For each $e\in E(B)$, we have $\varepsilon (e)=0$.
Therefore, by Theorem \ref{th00}, we have $^\partial{\varepsilon_{e}}(z)=2$. Consequently, by Proposition \ref{pro01}, we have $^\partial{\varepsilon_{B}}(z)=2^{e(B)}$.
\end{proof}

\begin{remark}
By Propositions \ref{pro06}, \ref{pro05}, and \ref{pro00}(2), 
if $H$ is a hypertree, then $^\partial{\varepsilon_{H}}(z)=2^{e(H)}$.
\end{remark}

\section*{Acknowledgements}
This work is supported by NSFC (Nos. 12471326, 12101600).



\begin{thebibliography}{9}

\bibitem{QCYC}  Q. Chen and Y. Chen,  Parallel edges in ribbon graphs and interpolating behavior of partial-duality polynomials, \emph{European J. Combin. } \textbf{102} (2022) 103492. 

\bibitem{CG} S. Chmutov, Generalized duality for graphs on surfaces and the signed Bollob\'as-Riordan polynomial,
\emph{J.  Combin. Theory Ser. B} \textbf{99} (2009) 617--638.

\bibitem{SCFV}  S. Chmutov and F. Vignes-Tourneret, On a conjecture of Gross, Mansour and Tucker, \emph{European J. Combin. } \textbf{97} (2021) 103368.

\bibitem{SCFV2} S. Chmutov and F. Vignes-Tourneret, Partial duality of hypermaps, \emph{Arnold Math. J.} \textbf{8} (2022) 445--468.

\bibitem{EM}  J. A. Ellis-Monaghan and I. Moffatt, Graphs on surfaces, Springer New York, 2013.

\bibitem{GMT}  J. L. Gross, T. Mansour and T. W. Tucker, Partial duality for ribbon graphs, I: Distributions, \emph{European J. Combin. } \textbf{86} (2020) 103084.

\bibitem{GMT2}  J. L. Gross, T. Mansour and T. W. Tucker, Partial duality for ribbon graphs, II: partial-twuality polynomials and monodromy computations, \emph{European J. Combin. } \textbf{95} (2021) 103329.

\bibitem{GMT3}  J. L. Gross, T. Mansour and T. W. Tucker,  Partial duality for ribbon graphs, III: a Gray code algorithm for enumeration, \emph{J. Algebraic Combin.} \textbf{54} (2021)  1119--1135.

\bibitem{BSM} B. Smith, Matroids, Eulerian graphs and topological analogues of the Tutte polynomial, Ph.D. thesis, Royal Holloway, University of London (2018).

\bibitem{QYXJ} Q. Yan and X. Jin, Counterexamples to a conjecture by Gross, Mansour and Tucker on partial-dual genus polynomials of ribbon graphs, \emph{European J. Combin. } \textbf{93} (2021) 103285.
    
\bibitem{QYXJ2} Q. Yan and X. Jin, Partial-dual genus polynomials and signed intersection graphs, \emph{Forum Math. Sigma} \textbf{10} (2022) e69.
\end{thebibliography}
\end{document}